\documentclass[12pt,reqno]{amsart}
\usepackage[left=3cm,top=2cm,right=3cm,bottom=2cm]{geometry}
\usepackage{ltxcmds}
\usepackage{iftex}
\usepackage{pdftexcmds}
\usepackage{amssymb}
\usepackage{amsbsy}
\usepackage{epsfig}
\usepackage{tikz}
\usepackage{verbatim}
\usepackage[normalem]{ulem}
\usepackage[pdftex]{hyperref}
\hypersetup{colorlinks=true,linkcolor=blue,citecolor=blue,breaklinks = true}

\tikzstyle{vertex} = [fill,shape=circle,node distance=80pt]
\tikzstyle{edge} = [fill,opacity=.5,fill opacity=.5,line cap=round, line join=round, line width=40pt]
\tikzstyle{elabel} =  [fill,shape=circle,node distance=30pt]

\pgfdeclarelayer{background}
\pgfsetlayers{background,main}

\begin{document}
\title{Energies of Hypergraphs}	


\author[K. Cardoso]{Kau\^e Cardoso} \address{Instituto Federal do Rio Grande do Sul - Campus Feliz, CEP 95770-000, Feliz, RS, Brazil}
\email{\tt kaue.cardoso@feliz.ifrs.edu.br}
 	
\author[V.Trevisan]{Vilmar Trevisan} \address{Instituto de Matem\'atica e Estat\'{\i}stica, UFRGS,  CEP 91509--900, Porto Alegre, RS, Brazil}
\email{\tt trevisan@mat.ufrgs.br}

\pdfpagewidth 8.5 in \pdfpageheight 11 in

\newcommand{\h}{\mathcal{H}}
\newcommand{\g}{\mathcal{G}}
\newcommand{\A}{\mathbf{A}}
\newcommand{\Q}{\mathbf{Q}}
\newcommand{\B}{\mathbf{B}}
\newcommand{\C}{\mathbf{C}}
\newcommand{\D}{\mathbf{D}}
\newcommand{\M}{\mathbf{M}}
\newcommand{\N}{\mathbf{N}}
\newcommand{\E}{\mathtt{E}}
\newcommand{\BE}{\mathtt{BE}}
\newcommand{\QE}{\mathtt{QE}}
\newcommand{\lin}{\mathcal{L}}
\newcommand{\cli}{\mathcal{C}}
\newcommand{\s}{\mathcal{S}}
\newcommand{\x}{\mathbf{x}}
\newcommand{\y}{\mathbf{y}}
\newcommand{\z}{\mathbf{z}}
\newcommand{\Ah}{\mathbf{A}(\mathcal{H})}

\theoremstyle{plain}
\newtheorem{Teo}{Theorem}
\newtheorem{Lem}[Teo]{Lemma}
\newtheorem{Pro}[Teo]{Proposition}
\newtheorem{Cor}[Teo]{Corollary}

\theoremstyle{definition}
\newtheorem{Def}{Definition}[section]
\newtheorem{Afi}[Def]{Affirmation}
\newtheorem{Que}[Def]{Question}
\newtheorem{Exe}[Def]{Example}
\newtheorem{Obs}[Def]{Remark}

\maketitle

\begin{abstract}
In this paper, we study energies associated with hypergraphs. More precisely, we obtain results for the incidence and the singless Laplacian energies of uniform hypergraphs. In particular, we obtain bounds for the incidence energy as functions of well known parameters, such as maximum degree, Zagreb index and spectral radius. We also relate the incidence and signless Laplacian energies of a hypergraph with the adjacency energies of its subdivision graph and line multigraph, respectively. In addition, we compute the signless Laplacian energy for the class of the power hypergraphs.\newline


\end{abstract}

\section{Introduction}
The study of molecular orbital energy levels of $\pi$-electrons in conjugated hydrocarbons may be seen as one of the oldest applications of spectral graph theory (see \cite{energy-book}). Research on this topic can be traced back to the 1930s \cite{energy-huckel}. In those studies, graphs were used to represent hydrocarbon molecules and it was shown that an approximation of the total $\pi$-electron energy may be  computed from the eigenvalues of the graph. Based on this chemical concept, in 1977 Gutman \cite{energy-Gutman} defined graph energy, starting a new line of research within the spectral graph theory community. In 2007, Nikiforov \cite{energy-niki} extended the concept of graph energy to matrices. For a matrix $\M$, its \textit{energy} $\E(\M)$, is defined as the sum of its singular values. From this work, other energies associated with graphs emerged, such as incidence energy \cite{incidence1} in 2009 and signless Laplacian energy \cite{energy-lap1} in 2010.

Regarding hypergraphs, a natural way to define  energy is to associate a hypegraph with a matrix $\M$ and then, using Nikiforov's definition, say that its energy is $\E(\M)$. It is worth mentioning that we have found no record of this natural extension in the literature. Perhaps the main reason for this lack of results is the fact that in 2012, Cooper and Dutle \cite{Cooper} proposed the study of hypergraphs through tensors, and this new approach has been widely accepted by researchers of this area. However, to  obtain  eigenvalues of tensors has a high computational and theoretical cost, so the definition of energy does not seem so natural in that setting. In this regard, we see that  the study of hypergraphs via matrices still has its place. Indeed, the first attempts to study spectral theory of hypergraphs were done using matrices \cite{Feng} and it is worth pointing out that more recently, some authors have renewed the interest to study matrix representations of hypergraphs, as in \cite{Banerjee, Kaue-lap, Reff20192}.

Following this trend, we propose in this note the study of hypergraph energies from their matrix representations. More precisely, we define and study two energies associated with hypergraphs. First, suppose $\h$ is a hypergraph and $\B$ is its incidence matrix,  we define its \textit{incidence energy} $\BE(\h)$ as the energy of $\B$.

Here we prove some interesting properties for the incidence energy as for example, we show that if $\BE(\h)$ is a rational number, then it is an integer, and also if $k$ or $m$ are even, then $\BE(\h)$ is even. In addition, we obtain several lower and upper bounds, and a Nordhaus–Gaddum type result relating $\BE(\h)$ to important parameters. A surprising result proved here relates the incident energy of a hypergraph to the adjacency energy of a graph as follows. The subdivision graph $\s(\h)$ is obtained by adding a new vertex to each hyperedge $e$ and make it adjacent to all vertices of $e$.

\begin{Teo}\label{teo:grafosub}If $\h$ is a uniform hypergraph, then $\BE(\h) = \frac{1}{2}\E(\A_\s)$, where $\A_\s$ is the adjacency matrix of $\s(\h)$.
\end{Teo}

The \textit{signless Laplacian matrix} of a hypergraph is defined (see \cite{Kaue-lap}) as $\Q = \B\B^T$. So, we define its \textit{signless Laplacian energy} as $\QE(\h) = \E\left( \Q- d(\h)\mathbf{I}\right)$. Here we make a detailed study of the relationship of this parameter to the adjacency energy of the line multigraph associated with $\h$. We are also able to bound the variation in energy when we add a new edge to the hypergraph.

As a particular case, we also study this energy for the class of power hypergraphs (see definition in Section \ref{sec:power}). We prove that, if a sufficiently large number of new vertices is added to each edge of the hypergraph, then it is possible to determine its signless Laplacian energy even without knowing its spectrum.

The paper is organized as follows. In Section \ref{sec:pre},
we present some basic definitions about hypergraphs and matrices. In Section
\ref{sec:incidence}, we study the incidence energy, extending many classical results of this energy to the context of uniform hypergraphs. In Section \ref{sec:laplacian}, we study the signless Laplacian energy of a hypergraph, relating this spectral parameter with the adjacency energy of the line multigraph. In Section \ref{sec:power}, we study the signless Laplacian energy of a power hypergraph.

\section{Preliminaries}\label{sec:pre}
In this section, we shall present some basic definitions about hypergraphs and matrices, as well as terminology, notation and concepts that will be useful in our proofs.
\vspace{0cm}

A \textit{hypergraph} $\h=(V,E)$ is a pair composed by a set of vertices $V(\h)$ and a set of (hyper)edges $E(\h)\subseteq 2^V$, where $2^V$ is the power set of $V$. $\h$ is said to be a $k$-\textit{uniform} (or a $k$-graph) for $k \geq 2$, if all edges have cardinality $k$. Let $\mathcal{H}=(V,E)$ and $\mathcal{H}'=(V',E')$ be hypergraphs, if $V'\subseteq V$ and $E'\subseteq E$, then $\mathcal{H}'$ is a \textit{subgraph} of $\h$. The \textit{complete $k$-graph} $\mathcal{K}_n$ on $n$ vertices, is a hypergraph, such that any subset of $V(\mathcal{K}_n)$ with $k$ vertices is an edge in $E(\mathcal{K}_n)$. The \textit{complement} of a $k$-graph $\h=(V,E)$ is the $k$-uniform hypergraph  $\overline{\h} = (\overline{V}, \overline{E})$, where $V = \overline{V}$ and $\overline{E} = E(\mathcal{K}_n) \smallsetminus E$.
\vspace{0.5cm}

The \textit{edge neighborhood} of a vertex $v\in V$, denoted by $E_{[v]}$, is the set of all edges that contains $v$. More precisely, $\;E_{[v]}=\{e:\;v\in e\in E\}$. The \textit{degree} of a vertex $v\in V$, denoted by $d(v)$, is the number of edges that contain $v$. More precisely, $\;d(v) = |E_{[v]}|$. A hypergraph is $r$-\textit{regular} if $d(v) = r$ for all $v \in V$. We define the \textit{maximum}, \textit{minimum} and \textit{average} degrees, respectively, as
\[\Delta(\h) = \max_{v \in V}\{d(v)\}, \quad \delta(\h) = \min_{v \in V}\{d(v)\}, \quad d(\h) = \frac{1}{n}\sum_{v \in V}d(v).\]

For a hypergraph $\h$, its \textit{line multigraph} $\lin(\h)$ is obtained by transforming the hyperedges of $\h$ in its vertices, and the number of edges between two vertices of this multigraph is equal the number of vertices in common in the two respective hyperedges. The \textit{clique multigraph} $\cli(\h)$, is obtained by transforming the vertices of $\h$ in its vertices. The number of edges between two vertices of this multigraph is equal the number of hyperedges containing them in $\h$. For more details see \cite{Kaue-lap}.

\begin{Exe}
	The clique and line multigraphs from $\h=(\{1,\ldots,5\},
	\;\{123,145,345\})$, are illustrate in Figure \ref{fig:multigrafos}.
	\begin{figure}[h!]	
		\centering

		\begin{tikzpicture}	
		\node[draw,circle,fill=black,label=below:,label=above:\(1\)] (v1) at (0,0) {};
		\node[draw,circle,fill=black,label=below:,label=above:\(2\)] (v2) at (1.5,2) {};
		\node[draw,circle,fill=black,label=below:,label=above:\(3\)] (v3) at (3,0) {};
		\node[draw,circle,fill=black,label=below:,label=above:\(4\)] (v4) at (1.5,0) {};
		\node[draw,circle,fill=black,label=below:,label=above:\(5\)] (v5) at (1.5,1) {};
		
		\path
		(v1) edge node[below]{} (v2)
		(v2) edge node[below]{} (v3)
		(v3) edge [bend left] node[below]{} (v1)
		(v1) edge node[below]{} (v4)
		(v1) edge node[below]{} (v5)
		(v4) edge [bend left] node[below]{} (v5)
		(v4) edge [bend right] node[below]{} (v5)
		(v4) edge node[below]{} (v3)
		(v3) edge node[below]{} (v5);
		\end{tikzpicture}
		\begin{tikzpicture}	
		\node[draw,circle,fill=white, label=above:\(123\)] (a) at (1.5, 2) {};
		\node[draw,circle,fill=white, label=above:\(145\)] (b) at (0, 0) {};
		\node[draw,circle,fill=white, label=above:\(345\)] (c) at (3, 0) {};

		\path
		(a) edge node[below]{} (b)
		(a) edge node[below]{} (c)
		(b) edge [bend left] node[below]{} (c)
		(b) edge [bend right] node[below]{} (c);
		\end{tikzpicture}
		\caption{~Clique $\cli(\h)$ and line $\lin(\h)$ multigraphs.}\label{fig:multigrafos}	
	\end{figure}
\end{Exe}

Let $\M$ be matrix. The \emph{singular values} of $\M$ are the square roots of the eigenvalues of the matrix $\M\M^T$. The \textit{rank} of $\M$ is defined as the number of non zero singular values (counting multiplicities). If $\M$ is a square matrix with $n$ rows, we denote its \textit{characteristic polynomial} by $P_\M(\lambda)=\det(\lambda\mathbf{I}_n-\M)$. Its eigenvalues will be denoted by $\lambda_1(\M)\geq\cdots\geq\lambda_n(\M)$. The \textit{spectral radius} $\rho(\M)$, is the largest modulus of an eigenvalue. Let $\h = (V, E)$ be a hypergraph. The \textit{incidence matrix} $\B(\h)$ is defined as the matrix of order $|V|\times|E|$, where $\;b(v,e) = 1$ if  $v \in e$ and
$\;b(v,e) = 0$ otherwise.

\begin{Lem}[Theorem 2, \cite{Kaue-lap}]\label{teo:multigrafo}	
	Let $\h$ be a $k$-graph, $\B$  its incidence matrix, $\D$ its degree matrix,
	$\A_\lin$ and $\A_\cli$ the adjacency matrices of its line and clique
	multigraphs, respectively. So, we have $\B^T\B = k\mathbf{I}+\A_\lin$,  and  $\;\B\B^T =\D+\A_\cli.$
\end{Lem}

For a non-empty subset of vertices $\alpha = \{v_1,\ldots,v_t\} \subset V$ and a vector $\x=(x_i)$ of dimension $n=|V|$, we denote $x(\alpha)=x_{v_1}+\cdots+x_{v_t}$. Recall that the signless Laplacian matrix is defined as $\Q = \B\B^T$, so we can write
\[(\Q\x)_u = (\D\x)_u+(\A_\cli\x)_u = d(u)x_u+\sum_{e \in E_{[u]}}x\left(e-\{u\}\right)  = \sum_{e \in E_{[u]}}x(e), \quad \forall u \in V(\h).\]

\section{Incidence energy}\label{sec:incidence}
In this section, we will study the incidence energy of a hypergraph, relating it to the adjacency energy of its subdivision graph. In addition, we obtain upper and lower bounds for this parameter. Many  results in this section are generalizations of incidence energy properties in the context of graphs, which can be found in  \cite{incidence2, incidence3, incidence1}.
\begin{Def}
	Let $\h$ be a $k$-graph with at least one edge and $\B$ its incidence matrix. The incidence energy of $\h$ is defined as the energy os its incidence matrix. More precisely $\BE(\h) = \E(\B)$. If $\h$ has no edge, then we define $\BE(\h) = 0$.
\end{Def}
Let $\h$ be a $k$-graph with $n$ vertices and $m$ edges, let $\Q$ be its signless Laplacian matrix and $\lin$ its line multigraph. We observe that, $\BE(\h) = \sum_{i=1}^n\sqrt{\lambda_i(\Q)}$, this occurs because $\B\B^T=\Q$. Similarly, $\BE(\h) = \sum_{i=1}^{m}\sqrt{k+\lambda_i(\A_\lin)}$, this occurs because $\B\B^T$ and $\B^T\B$ have the same non zero eigenvalues and, moreover, $\B^T\B=k\mathbf{I}+\A_\lin$.

\begin{Exe}\label{exe:comp}
We will determine the incidence energy of the complete $k$-graph. First, we notice that its eigenvalues are $\rho(\Q)=\frac{k(n-1)!}{(k-1)!(n-k)!}$ and $\lambda = \frac{(n-2)!}{(k-1)!(n-k-1)!}$ with multiplicity $n-1$. Therefore, the incidence energy of the complete $k$-graph is
	$$\BE(\mathcal{K}_n) = \sqrt{\frac{k(n-1)!}{(k-1)!(n-k)!}} + (n-1)\sqrt{\frac{(n-2)!}{(k-1)!(n-k-1)!}}.$$
\end{Exe}

\begin{Def}
	Let $\h$ be a $k$-graph. Its subdivision graph $\s(\h)$ is obtained as follows. For each hyperedge $e \in E(\h)$, add a new vertex $v_e$ and make it adjacent to all vertices of $e$.
\end{Def}

\begin{Exe}
	Let $\h$ be the $3$-graph with the following sets of vertices and edges
	$V = \{1,2,3,4\}$,  $E=\{123,234\}$. We illustrate it and its subdivision graph in Figure \ref{fig:grafosub}.
	
	\begin{figure}[h!]
		\centering
		\begin{tikzpicture}
		\node[draw,circle,fill=black,label=below:,label=above:\(1\)] (v1) at (0,0) {};
		\node[draw,circle,fill=black,label=below:,label=above:\(2\)] (v2) at (2,1) {};
		\node[draw,circle,fill=black,label=below:,label=above:\(3\)] (v3) at (2,-1) {};
		\node[draw,circle,fill=black,label=below:,label=above:\(4\)] (v4) at (4,0) {};

		\begin{pgfonlayer}{background}
		\draw[edge,color=gray] (v1) -- (v2) -- (v3) -- (v1);
		\draw[edge,color=gray] (v4) -- (v2) -- (v3) -- (v4);
		
		\end{pgfonlayer}
		\end{tikzpicture}
		\begin{tikzpicture}
		[scale=1,auto=left,every node/.style={circle,scale=0.9}]
		\node[draw,circle,fill=black,label=below:,label=above:\(1\)] (v1) at (0,0) {};
		\node[draw,circle,fill=black,label=below:,label=above:\(2\)] (v2) at (2,1) {};
		\node[draw,circle,fill=black,label=below:,label=above:\(3\)] (v3) at (2,-1) {};
		\node[draw,circle,fill=black,label=below:,label=above:\(4\)] (v4) at (4,0) {};
		\node[draw,circle,fill=white,label=below:] (va) at (1,0) {};
		\node[draw,circle,fill=white,label=below:] (vb) at (3,0) {};

		\path
		(v1) edge node[below]{} (va)
		(v2) edge node[below]{} (va)
		(v3) edge node[below]{} (va)
		(v4) edge node[below]{} (vb)
		(v2) edge node[below]{} (vb)
		(v3) edge node[below]{} (vb);
		\end{tikzpicture}
		\caption{~The hypergraph $\h$ and its subdivision graph.}\label{fig:grafosub}
	\end{figure}
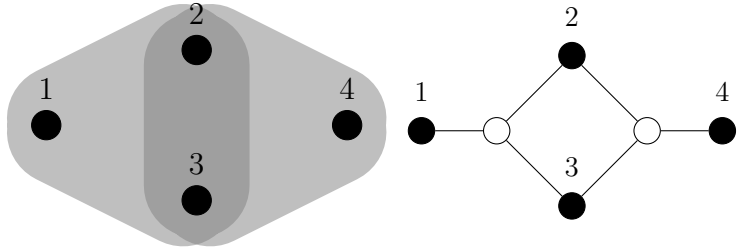
\end{Exe}

\begin{Obs}
	Informally we may see that the subdivision graph of $\h$ transforms each hyperedge into a star with $k+1$ vertices. If $\h$ has $n$ vertices and $m$ edges, then $\s(\h)$ is a bipartite graph with $n+m$ vertices and $km$ edges. Also, if $\B$ is the incidence matrix of $\h$, then the adjacency matrix of $\s(\h)$ is given by
	$\A_\s = \left( \begin{smallmatrix}\mathbf{0} & \B\\ \;\B^T & \mathbf{0}\end{smallmatrix}\right).$
\end{Obs}

\begin{Pro}\label{lem:grafosub}
Let $\h$ be a $k$-graph on $n$ vertices and $m$ edges, and $\s$ be its subdivision graph. The set  $\{\lambda_1,\ldots,\lambda_t,0^{n-t}\}$ is the spectrum of $\Q$ if and only if the spectrum of $\A_\s$ is $\{\pm\sqrt{\lambda_1},\ldots,\pm\sqrt{\lambda_t},0^{m+n-2t}\}$.
\end{Pro}
\begin{proof}
First, we observe that $\lambda$ is an eigenvalue of $\Q$ if and only if $\sigma = \sqrt{\lambda}$ is a singular value of $\B$. Let $\x$ and $\y$ be the singular vectors of $\sigma$, such that $\;\B\x = \sigma\y\;$ and $\;\B^T\y = \sigma\x$. We define a vector $\z$ of dimension $m+n$ by $z_i = y_i$ if $1 \leq i \leq n$ and $z_j = x_j$ if $n+1 \leq j \leq m+n$. We have
	$$\A_\s\z = \left[ \begin{matrix}\B\x\\\B^T\y\end{matrix}\right] = \left[ \begin{matrix}\sigma\y\\\sigma\x\end{matrix}\right] = \sigma\z, \quad \Rightarrow \quad \sigma \textrm{ is an eigenvalue of } \s(\h).$$
	
	Let $\sigma$ be a positive eigenvalue of $\s(\h)$ and $\z$ be an eigenvector of $\sigma$. We define vectors $\y$ and $\x$ of dimensions $n$ and $m$ respectively, by $y_i = z_i$ if $1\leq i \leq m$ and $x_i = z_{m+i}$ if $1\leq i \leq n$. We have
	$$\left[ \begin{matrix}\B\x\\\B^T\y\end{matrix}\right] = \A_\s\z = \sigma\z = \left[ \begin{matrix}\sigma\y\\\sigma\x\end{matrix}\right], \quad \Rightarrow \quad \sigma \textrm{ is an singular value of } \B.$$
	
	Finally, if $\sigma$ is a positive eigenvalue of $\s(\h)$, then $-\sigma$ is also an eigenvalue of $\s(\h)$, and with the same multiplicity. In fact, if $\z = (\y,\x)$ is an eigenvector of $\sigma$, we define $\tilde{\z} = (\y,-\x)$. We have
	$$\A_\s\tilde{\z} = \left[ \begin{matrix}\B(-\x)\\\B^T\y\end{matrix}\right] = \left[ \begin{matrix}-\sigma\y\\\sigma\x\end{matrix}\right] = -\sigma\left[ \begin{matrix}\y\\-\x\end{matrix}\right] = -\sigma\tilde{\z}.$$
	Under these conditions, we conclude that the set of all nonzero eigenvalues of $\s(\h)$ is  $\{\pm\sqrt{\lambda_1},\ldots,\pm\sqrt{\lambda_t}\}$.
\end{proof}

\vspace{0.1cm}\noindent \textbf{Theorem \ref{teo:grafosub}.} \textit{If $\h$ is a $k$-graph, then $\BE(\h) = \frac{1}{2}\E(\A_\s)$.}
\begin{proof}
	Let $\lambda_1,\ldots,\lambda_t$ be all positive eigenvalues of $\Q(\h)$. By Proposition \ref{lem:grafosub}, we have
	$$\E(\A_\s) = |\sqrt{\lambda_1}|+\cdots+|\sqrt{\lambda_t}|+|-\sqrt{\lambda_1}|+\cdots+|-\sqrt{\lambda_t}| = 2\BE(\h).$$
	Therefore, the result follows.
\end{proof}

\begin{Lem}[Lemma 2, \cite{energy-ref2}]\label{lem:rank}
	If $\g$ is a graph, then $\E(\A_\g) \geq \mathrm{rank}(\A_\g)$.
\end{Lem}

\begin{Cor}
	If $\h$ is a $k$-graph with incidence matrix $\B$, then $\BE(\h) \geq \mathrm{rank}(\B)$.
\end{Cor}
\begin{proof}
	We observe that $\BE(\h) = \frac{1}{2}\E(\A_\s) \geq \frac{1}{2}\mathrm{rank}(\A_\s) = \mathrm{rank}(\B).$ The first equality is given by Theorem \ref{teo:grafosub}, while the inequality is given by Lemma \ref{lem:rank} and the last equality is from Proposition \ref{lem:grafosub}.
\end{proof}

\begin{Lem}[Lemmas 1 and 2, \cite{energy-odd2}]\label{lem:soma-prod}
	Let $\g_1$ and $\g_2$ be graphs. If $\mu_1$  and $\mu_2$ are eigenvalues of $\g_1$ and $\g_2$ respectively, then there are graphs $\g_+$ and $\g_\times$ such that, $\mu_1+\mu_2$ and $\mu_1\cdot\mu_2$ are eigenvalues of $\g_+$ and $\g_\times$, respectively.
\end{Lem}

\begin{Lem}[Lemma 3, \cite{energy-odd2}]\label{lem:inter}
	If an eigenvalue of a graph is rational, then it is an integer.
\end{Lem}

\begin{Lem}[See \cite{energy-odd1}]\label{teo:never-odd}
	If the energy of a graph is rational, then it is an even integer.
\end{Lem}

\begin{Teo}
	Let $\h$ be a $k$-graph with $m$ edges. If its incidence energy $\BE(\h)$ is a rational number, then it is an integer. Moreover:
	\begin{enumerate}
		\item[(a)] If $k$ is even, then $\BE(\h)$ is also even.
		\item[(b)] If $k$ is odd, then $\BE(\h)$ and $m$ have the same parity.
	\end{enumerate}
\end{Teo}
\begin{proof}
If $\BE$ is rational, then by Theorem \ref{teo:grafosub}, $\E(\A_\s)$ is rational and by Lemma~\ref{teo:never-odd}, we know that $\E (\A_\s)$ is even, thus $\BE = \frac{1}{2}\E (\A_\s)$ is integer. Now, we notice
	$$(\BE)^2 = (\sigma_1+\cdots+\sigma_n)^2 = \sum_{i=1}^n\sigma_i^2 + 2\!\!\sum_{1\leq i < j \leq n}\sigma_i\sigma_j = km + 2\!\!\sum_{1\leq i < j \leq n}\sigma_i\sigma_j.$$
	If $\BE$ is integer, then $\sum\sigma_i\sigma_j$ must be rational. By Lemma \ref{lem:soma-prod} this sum is an eigenvalue of a graph. By Lemma \ref{lem:inter} this sum must be an integer, denote $p = \sum\sigma_i\sigma_j$. That is, $(\BE)^2 = km +2p$. Therefore, if $k$ or $m$ are even, then $\BE$ is even too, otherwise $\BE$ is odd.
\end{proof}

\begin{Lem}[See \cite{spectraofgraphs}]\label{lem:entralacamento} If $\mathbf{R}$ and $\mathbf{S}$ are symmetric matrices with $n$ rows, then 	$$\lambda_i(\mathbf{R}+\mathbf{S}) \geq \max\{\lambda_i(\mathbf{R}), \lambda_i(\mathbf{S})\}, \quad \forall i=1,\ldots,n.$$
\end{Lem}

\begin{Pro}
	Let $\h=(V,E)$ be a $k$-graph. For each edge $e \in E(\h)$, we have
	$$\BE(\h) > \BE(\h-e).$$
\end{Pro}
\begin{proof}
	Let $\h-e=(V,E\smallsetminus\{e\})$ and $\h[e] = (V,\{e\})$ be subgraphs of $\h$. We see that $\;\Q(\h) = \Q(\h-e) + \Q(\h[e]).\;$ By Lemma~\ref{lem:entralacamento}, we have $\lambda_i(\Q(\h)) \geq \lambda_i(\Q(\h-e))$ for each $i \in V$, so $\BE(\h) \geq \BE(\h-e)$. Now we suppose, by way of	contradiction, that  $\lambda_i(\Q(\h)) = \lambda_i(\Q(\h-e))$ for all $i \in V$, thus $\mathrm{Tr}(\Q(\h)) = \mathrm{Tr}(\Q(\h-e))$, so $\mathrm{Tr}(\Q(\h[e])) = 0$, which is a contradiction. Therefore the inequality must be strict.
\end{proof}

\begin{Cor}
	The complete $k$-graph $\mathcal{K}_n$ has the highest incidence energy between the $k$-graphs with $ n $ vertices. That is, if $ \h $ is a $k$-graph with $ n $ vertices, then
	$$\BE(\h) \leq \sqrt{\frac{k(n-1)!}{(k-1)!(n-1)!}} + (n-1)\sqrt{\frac{(n-2)!}{(k-1)!(n-k-1)!}}.$$
\end{Cor}

\subsection{Bounds for the incidence energy}

\begin{Teo}
	If $\h$ is a $k$-graph with $n$ vertices and $m$ edges, then
	$$\sqrt{km} \leq \BE(\h) \leq \sqrt{kmn}.$$
Moreover, the first inequality is attained if and only if $\h$ has at most one edge, while the second inequality is attained if and only if the hypergraph has no edges.
\end{Teo}
\begin{proof}
	First, we notice that
	$$\BE = \sum_{i=1}^n\sigma_i \geq \sqrt{\sum_{i=1}^n\sigma^2_i} = \sqrt{km}.$$
	This inequality  can only be attained if at most one singular value is nonzero. That is $\mathrm{rank}(\B) \leq 1$, so $\h$ must not have more than one edge.
	
	 Now, using the Cauchy-Schwarz inequality, we obtain
	$$\BE = \sum_{i=1}^n\sigma_i \leq \sqrt{n\sum_{i=1}^n\sigma^2_i}= \sqrt{kmn}.$$
	However, this inequality is attained when all singular values are equal. Thus, $ \B\B^T = \sigma \mathbf{I} $, so $ \A_\cli = \mathbf{0} $, therefore $ \h $ should have no edges.
\end{proof}

\begin{Teo}
	Let $\h$ be a $k$-graph with $m$ edges. If $\mathrm{rank}(\B) = r$, then $$\BE(\h) \leq \sqrt{kmr} \leq \sqrt{k}m.$$
	Equality holds if  and only if $\h$ is formed by disjoint edges.
\end{Teo}
\begin{proof}
Using Cauchy-Schwarz inequality, we notice that
	\begin{eqnarray}\notag\BE &=& \sum_{i=1}^{m}\sqrt{k+\lambda_i(\lin)} = \sum_{i=1}^{r}\sqrt{k+\lambda_i(\lin)} \leq \sqrt{r \sum_{i=1}^{r}(k+\lambda_i(\lin))}\\\notag &=& \sqrt{r(kr+\mathrm{Tr}(\A_\lin)+k(m-r))} = \sqrt{kmr} \leq \sqrt{km^2} = \sqrt{k}m. \end{eqnarray}
Further, equality holds only when all eigenvalues of $\lin(\h)$ are equal. But as $\mathrm{Tr}(\A_\lin) = 0$, so $\A_\lin = \mathbf{0}$. That is, $\lin$ must have only isolated vertices, therefore $\h$ must have disjoint edges.
\end{proof}

\begin{Lem}\label{lem:cota-rho-index}
	If $\h$ is a $k$-graph on $n$ vertices and $m$ edges, then
	$$\BE(\h) \leq  \sqrt{\rho} + \sqrt{(n-1)(km-\rho)}.$$
	Also, if $\h$ is complete, then the equality holds.
\end{Lem}
\begin{proof}
	We observe that, $\displaystyle \sum_{i=2}^n\lambda_i(\Q) = km-\rho(\Q)$ and thus, by Cauchy-Schwarz inequality we have
	$$\BE = \sqrt{\rho}+\sum_{i=2}^n\sqrt{\lambda_i} \leq \sqrt{\rho}+\sqrt{(n-1)\sum_{i=2}^n\lambda_i} = \sqrt{\rho} + \sqrt{(n-1)(km-\rho)}.$$
	To finish the proof, we observe that for a complete $k$-graph, all signless Laplacian eigenvalues distinct from the spectral radius, are equal to $\frac{km-\rho}{n-1}$, see Example \ref{exe:comp}.
\end{proof}

Let $\h$ be a hypergraph. Its \textit{Zagreb index} is defined as the sum of the squares of the degrees of its vertices. More precisely $$Z(\h) = \sum_{v \in V(\h)}d(v)^2.$$ This is an important parameter in graph theory, having chemistry applications, \cite{energy-Zagreb}. We define the following auxiliary value $\mathfrak{Z}(\h) =k\sqrt{\frac{1}{n}Z(\h)}$.

\begin{Lem}[Theorem 13, \cite{Kaue-lap}]\label{teo:qregular}
	Let $\h$ be a connected $k$-graph on $n$ vertices and $\Q$ its signless Laplacian matrix. The hypergraph $\h$ is regular if and only if $\x=\left(\frac{1}{\sqrt{n}},\ldots,\frac{1}{\sqrt{n}}\right)$, is an eigenvector from $\rho(\Q)$.
\end{Lem}

\begin{Lem}
Let $\h$ be a $k$-graph. If $\Q$ is its signless Laplacian matrix, then $\rho(\Q)~\geq~\mathfrak{Z}(\h)$. Equality holds if and only if $\h$ is regular.
\end{Lem}
\begin{proof}
	Let $\y = (\frac{1}{\sqrt{n}},\ldots,\frac{1}{\sqrt{n}})$, so we have
	$$\rho(\Q) = \sqrt{\rho(\Q^2)} \geq \sqrt{\y^T\Q^2\y} = \sqrt{\sum_{v \in V}\frac{(kd(v))^2}{n}} = k\sqrt{\frac{1}{n}Z(\h)}.$$
We notice that the equality holds if and only if $\y = (\frac{1}{\sqrt{n}},\ldots,\frac{1}{\sqrt{n}})$ is an eigenvector of $\rho(\Q)$. By Lemma \ref{teo:qregular} it occurs only if the hypergraph is regular.
\end{proof}

The following result improves the bound of Lemma \ref{lem:cota-rho-index}.

\begin{Teo}\label{pro:cota-indice-z}
	If $\h$ is a $k$-graph on $n$ vertices and $m$ edges, then
	$$\BE(\h) \leq \sqrt{\mathfrak{Z}(\h)} + \sqrt{(n-1)(km-\mathfrak{Z}(\h))}.$$
	Also, if $\h$ is complete then equality holds.
\end{Teo}
\begin{proof}
First, notice that $f(x) = \sqrt{x} + \sqrt{(n-1)(km-x)}$ is decreasing if $x > \frac{km}{n}$. To prove it, we observe that, if $x > \frac{km}{n}$, then $f'(x) < 0$. Now, we notice
	$$\sum_{v \in V}d(v) \leq \sqrt{n\sum_{v \in V}d^2(v)} \quad \Rightarrow \quad \frac{km}{n} \leq \sqrt{\frac{1}{n}\sum_{v \in V}d^2(v)} < \mathfrak{Z}(\h).$$
	As $\rho(\Q) \geq \mathfrak{Z}(\h)$, so the result follows.
\end{proof}

\begin{Lem}\label{lem:traco}
	If $\h$ is a $k$-graph with $n$ vertices, then
	$$\sum_{i=1}^n\lambda_i^2(\Q) \leq kZ(\h).$$
	The equality holds if and only if $\h$ is formed by disjoint edges.
\end{Lem}
\begin{proof}
	By Lemma \ref{teo:multigrafo}, we have $\Q = \D+\A_\cli$, so $\mathrm{Tr}(\Q^2) = Z(\h) +  \mathrm{Tr}(\A^2_\cli)$. Now, we notice that
	$$\mathrm{Tr}(\A^2_\cli) = \sum_{i=1}^n\sum_{j=1}^n a_{ij}^2\stackrel{(\#)}{\leq} \sum_{i=1}^n\left( d(i)\sum_{j=1}^n a_{ij}\right)= (k-1)\sum_{i=1}^nd(i)^2 = (k-1)Z(\h).$$
	The inequality $(\#)$ is true, because $a_{ij}$ is the number of edges containing the vertices $i$ and $j$, so $a_{ij} \leq \min\{d(i),d(j)\}$. Therefore, we conclude that $\sum_{i=1}^n\lambda_i^2(\Q) = \mathrm{Tr}(\Q^2) \leq kZ(\h).$	
	
	Now observe that, the equality in $(\#)$ is achieved only for hypergraphs with the following property: If two vertices $i$ and $j$ are neighbors then they are contained in the same edges. But it is possible only in hypergraphs with disjoint edges and possibly some isolated vertices. 	
\end{proof}

\begin{Lem}[Equation 12, \cite{incidence3}]\label{lem:seq-int}
	If $a_1,\ldots,a_s$ is a sequence of non negative integers, then
	$$\sum_{i=1}^sa_i \geq \sqrt{\frac{\left(\sum_{i=1}^s a_i^2 \right)^3 }{\sum_{i=1}^s a_i^4 }}.$$
	The equality holds if and only if all positive elements of the sequence are equal.
\end{Lem}

\begin{Pro}\label{pro:cotainf}
	If $\h$ is a $k$-graph with $n$ vertices and $m$ edges, then
	$$\BE(\h) \geq \sqrt{\frac{(km)^3}{kZ(\h)}} \geq \frac{\sqrt{k}m}{\sqrt{\Delta}}.$$
	The first equality holds if  and only if $\h$ is formed by disjoint edges. The second equality holds if and only if $\h$ is formed by disjoint edges without isolated vertices or it has no edges.
\end{Pro}
\begin{proof}
	By Lemmas \ref{lem:traco} and \ref{lem:seq-int}, we have
	$$\BE(\h) = \sum_{i=1}^n\sqrt{\lambda_i(\Q)} \stackrel{(*)}{\geq} \sqrt{\frac{\left(\sum_{i=1}^n(\sqrt{\lambda_i(\Q)})^2 \right)^3}{\sum_{i=1}^n(\sqrt{\lambda_i(\Q)})^4}} \stackrel{(**)}{\geq} \sqrt{\frac{(km)^3}{kZ(\h)}}.$$
	Now, we notice that
	$$Z(\h) = \sum_{i=1}^nd(i)^2 \leq \Delta\sum_{i=1}^nd(i) = \Delta km,$$
	therefore
	$$\BE(\h) \stackrel{(***)}{\geq} \frac{\sqrt{k}m}{\sqrt{\Delta}}.$$
	Finally, we notice that the equality in $(*)$ occur if and only if all positive eigenvalues are equal. That is, when the hypergraph is formed by disjoint edges. The equality in $(**)$ is achieved under the same conditions of $(*)$. The equality in $(***)$ occur if and only if the equality in $(*)$ occur and the hypergraph is regular. That is, when the hypergraph is formed by disjoint edges without isolated vertices or it has only isolated vertices.
\end{proof}

\begin{Lem}[Corollary 16, \cite{Kaue-lap}]\label{lem:graumed}
	If $\h$ is a $k$-graph, then $kd(\h) \leq \rho(\Q).$
\end{Lem}

\begin{Teo} Let $\h$ be a $k$-graph on $n$ vertices. If $\overline{\h}$ is its complement, then
	$$\frac{\sqrt{k}\binom{n}{k}}{\sqrt{\binom{n-1}{k-1}}} \leq \BE(\h) + \BE(\overline{\h}) \leq k\sqrt{\frac{2}{n}\binom{n}{k}} + \sqrt{\frac{2k(n-1)(n-k)}{n}\binom{n}{k}}.$$
	The first equality occur if and only if $\h$ has at most $k$ vertices and one edge.
\end{Teo}
\begin{proof}
	Suppose that $\h$ and $\overline{\h}$ have $m$ and $\overline{m}$ edges, respectively. We have $m+\overline{m} = \binom{n}{k}$ and, by Proposition \ref{pro:cotainf},
	\begin{eqnarray}\notag\BE(\h) + \BE(\overline{\h}) \geq \frac{\sqrt{k}m}{\sqrt{\Delta(\h)}} + \frac{\sqrt{k}\overline{m}}{\sqrt{\Delta(\overline{\h})}}\geq \frac{\sqrt{k}\binom{n}{k}}{\sqrt{\binom{n-1}{k-1}}}.\end{eqnarray}
	For the equality to be possible, we observe that $\h$ must be formed by disjoint edges without isolated vertices or only isolated vertices, as well as its complement. In addition, it must occur $\Delta(\h) = \binom{n-1}{k-1}$ or $m = 0$. That is, $\h$ must have at most $k$ vertices and one edge.

	Let $\lambda_1\geq\cdots \geq \lambda_n$ be all eigenvalues of $\Q(\h)$ and $\overline{\lambda_1}\geq\cdots \geq \overline{\lambda_n}$ be the eigenvalues of $\Q(\overline{\h})$. By Cauchy-Schwarz inequality, we have
	\begin{eqnarray}\BE(\h) + \BE(\overline{\h}) \notag &\leq& \sqrt{\lambda_1} + \sqrt{\overline{\lambda_1}} +\sqrt{(n-1)\sum_{i=2}^n\lambda_i} +\sqrt{(n-1)\sum_{i=2}^n\overline{\lambda_i}}\\
	&\leq& \sqrt{2(\lambda_1 + \overline{\lambda_1})} + \sqrt{2(n-1)\left[k\binom{n}{k} - (\lambda_1 + \overline{\lambda_1})\right]. }\label{eq:l1+l1}
	\end{eqnarray}
	We observe that the function $f(x) = \sqrt{2x} + \sqrt{2(n-1)(k\binom{n}{k}-x)}$ is decreasing for $x \geq \frac{k}{n}\binom{n}{k}$. Now, by Lemma \ref{lem:graumed} we have
	
	\begin{equation}\label{eq:l1+l2}
	\lambda_1 + \overline{\lambda_1} \geq \frac{k^2m}{n} + \frac{k^2\overline{m}}{n} = \frac{k^2}{n}\binom{n}{k} > \frac{k}{n}\binom{n}{k}.
	\end{equation}
	Changing $\lambda_1 + \overline{\lambda_1}$ by $\frac{k^2}{n}\binom{n}{k}$ in equation (\ref{eq:l1+l1}), we obtain the desired result.
\end{proof}

\section{Signless Laplacian energy}\label{sec:laplacian}
In this section, we will study the signless Laplacian energy of hypergraphs. This energy has already been well studied for graphs (see for example \cite{energy-lap2, energy-lap5, energy-lap6, energy-lap1, energy-lap7}). Our main result relates this energy to the adjacency energy of a line multigraph. For more details about the signless Laplacian matrix, see \cite{Kaue-lap}.

\begin{Def}
	Let $\h$ be a $k$-graph. We define its signless Laplacian energy as the energy of the matrix $\Q-d(\h)\mathbf{I}$, i.e. $\QE(\h) = \E(\Q-d(\h)\mathbf{I})$.
\end{Def}

\begin{Def}
	Let $\h$ be a $k$-graph. We define $\omega(\h)$ as the number of eigenvalues of $\Q$ greater than or equal to the average degree. More precisely, if $\lambda_1\geq\cdots\geq\lambda_n$ are the eigenvalues of $\Q$, then $\lambda_\omega \geq d(\h)$ and $\lambda_{\omega+1} < d(\h)$.
\end{Def}
\begin{Pro}
	If $\h$ is a $k$-graph, then $\QE(\h) = 2\sum_{i=1}^\omega\lambda_i -2\omega d(\h).$
\end{Pro}
\begin{proof}
	We notice that,
	\begin{eqnarray}
	\QE(\h) &=& \sum_{i=1}^n|\lambda_i-d(\h)| = \sum_{i=1}^\omega(\lambda_i-d(\h)) + \sum_{i=\omega+1}^n(d(\h)-\lambda_i)\notag\\
	&=& 2\sum_{i=1}^\omega(\lambda_i-d(h)) + \sum_{i=1}^n(d(\h)-\lambda_i) \notag\\
	&=& 2\sum_{i=1}^\omega\lambda_i - 2\omega d(\h) + \underbrace{nd(\h) -km}_{=0}.\notag
	\end{eqnarray}
	Therefore, the result follows.
\end{proof}

\begin{Lem}[Lemma 2.21, \cite{energy-lucelia}]\label{lem:energy-subgrafo}
	If $\M$ and $\mathbf{N}$ are square matrices, then
	$$\E(\M+\mathbf{N}) \leq \E(\M) + \E(\mathbf{N}), \quad |\E(\M) - \E(\mathbf{N})| \leq  \E(\M-\mathbf{N}). $$
\end{Lem}

\begin{Pro}
	Let $\h$ be a non complete $k$-graph. If $e \notin E(\h)$, then
	$$|\QE(\h+e) - \QE(\h)| \leq 2k-\frac{2k}{n}.$$
\end{Pro}
\begin{proof} First, we observe that
	\begin{eqnarray}
	|\QE(\h+e) - \QE(\h)| &=& \left| \E\left( \Q(\h+e) - \frac{k(m+1)}{n}\mathbf{I}\right)  - \E\left( \Q(\h) - \frac{km}{n}\mathbf{I}\right) \right| \notag \\
	&\leq& \E\left(\Q(\h+e) -  \Q(\h) -\frac{k}{n}\mathbf{I}\right). \notag
	\end{eqnarray}
	\noindent The inequality above follows from Lemma \ref{lem:energy-subgrafo}. Now, we observe that
	
	\[\M :=
	\Q(\h+e) -  \Q(\h) -\frac{k}{n}\mathbf{I} =
	\left[
	\begin{array}{cccc|cccc}
	1-\frac{k}{n} & 1 &\cdots&1&0&\cdots&0\\
	1 & 1-\frac{k}{n} &\cdots&1&0&\cdots&0\\
	\vdots&\vdots&\ddots&\vdots& \vdots&\vdots&\vdots\\
	1 & 1 &\cdots& 1-\frac{k}{n}&0&\cdots& 0\\
	\hline
	0&0&\cdots&0& -\frac{k}{n}&\cdots&0\\
	\vdots&\vdots&\vdots&\vdots& \vdots&\ddots&\vdots\\
	0&0&\cdots&0& 0&\cdots&-\frac{k}{n}\\
	\end{array}
	\right].
	\]
	That is, the eigenvalues of $\M$, are $k-\frac{k}{n}$ with multiplicity $1$ and $-\frac{k}{n}$ with multiplicity $n-1$, thus the energy of this matrix is
	$$\E \left(\Q(\h+e) -  \Q(\h) -\frac{k}{n}\mathbf{I}\right) =  k-\frac{k}{n} + (n-1)\frac{k}{n} = 2k-\frac{2k}{n}. $$
	Therefore, the result follows.		
\end{proof}

\begin{Teo} Let $\h$ be a $k$-graph with $n$ vertices and $m\geq1$ edges.
	\begin{itemize}
		\item[(a)] If $m=n$, then $\QE(\h) = \E(\A_\lin)$.
		\item[(b)] If $m < n$, then $\QE(\h)-\frac{2km(n-m)}{n}\leq \E(\A_\lin)< \QE(\h)$. Equality holds if and only if $\h$ has only isolated edges.
		\item[(c)] If $m > n$, then $\QE(\h) < \E(\A_\lin) < \QE(\h) + 2k(m-n)$.
	\end{itemize}
\end{Teo}
\begin{proof}
	To prove item $(a)$, we notice that
	$$\E(\A_\lin) = \E(\B^T\B - k\mathbf{I}) = \sum_{i=1}^m|\lambda_i(\B^T\B)-k|.$$
	Moreover,
	$$\QE(\h) = \E\left( \Q - \frac{km}{n}\mathbf{I}\right) = \E(\B\B^T - k\mathbf{I}) = \sum_{i=1}^n|\lambda_i(\B\B^T)-k|. $$
If $m=n$, then $\B\B^T$ and $\B^T\B$ have the same eigenvalues, so the equality is true.
	
For the first inequality of item $(b)$, we observe that if $i=1,\ldots,m$, then $\lambda_i(\B^T\B) = \lambda_i(\B\B^T)$ and if $m < i \leq n$, then $\lambda_i(\B\B^T) = 0$. We have	
	\begin{eqnarray}
	\notag \QE(\h) &=& \sum_{i=1}^n\left| \lambda_i - \frac{km}{n}\right|  = \sum_{i=1}^m\left| \lambda_i - \frac{km}{n}\right|  + \sum_{i=m+1}^n\left| \frac{km}{n}\right|\\
	\label{eq:soma-igualdade} &\leq& \sum_{i=1}^m\left|\lambda_i-k \right| + \sum_{i=1}^m\left|k - \frac{km}{n}\right| + \frac{km(n-m)}{n}\\
	\notag &=& \E(\A_\lin) + \frac{2km(n-m)}{n}.
	\end{eqnarray}
	\noindent The equality holds in (\ref{eq:soma-igualdade}) only if $\lambda_i-k$ and $k - \frac{km}{n}$ have the same sign. That is, $\lambda_i(\B^T\B) \geq k$, for each $i=1,\ldots,m$, thus
	$$\lambda_i(k\mathbf{I}+\A_\lin) = \lambda_i(\B^T\B) \geq k, \quad \Rightarrow \quad \lambda_i(\A_\lin) \geq 0.$$
	As $\mathrm{Tr}(\A_\lin) =0$, then all eigenvalues of $\A_\lin$ must be zeros. Therefore, $\A_\lin = \mathbf{0}$, so $\lin(\h)$ should have no edges, or equivalent, $\h$ should have only isolated edges.
	
Now, for the second inequality of item $(b)$, we observe that
	\begin{eqnarray}
	\notag \E(\A_\lin) &=& \sum_{i=1}^m|\lambda_i(\A_\lin)| = \sum_{i=1}^m|\lambda_i(\B^T\B) - k| = \sum_{i=1}^n|\lambda_i(\B\B^T) - k| - k(n-m)\\
	\notag &\leq& \sum_{i=1}^n\left| \lambda_i(\B\B^T) - \frac{km}{n}\right|  + \sum_{i=1}^n\left| k - \frac{km}{n}\right|  - k(n-m) = \QE(\h).
	\end{eqnarray}
	\noindent Similarly to the first part of this item, the equality could only occur if $\lambda_i - \frac{km}{n}$ is equal to or less than zero for all $i=1\ldots,n$, thus
	$$\lambda_i(\B^T\B)\leq \frac{km}{n} < k, \quad \Rightarrow \quad \lambda_i(\A_\lin) <0.$$
	As $\mathrm{Tr}(\A_\lin) =0$, this matrix cannot have all negative eigenvalues, so equality cannot be achieved.
	
	To prove the first inequality of item $ (c) $, notice that
	\begin{eqnarray}
	\notag\QE(\h) &=& \sum_{i=1}^n\left| \lambda_i(\B\B^T)-\frac{km}{n}\right|  = \sum_{i=1}^m\left| \lambda_i(\B^T\B)-\frac{km}{n}\right| -\frac{km(m-n)}{n}\\
	\notag &\leq& \sum_{i=1}^m\left| \lambda_i(\B^T\B)-k\right| + \sum_{i=1}^m\left| k -\frac{km}{n}\right|  -\frac{km(m-n)}{n}= \E(\A_\lin).
	\end{eqnarray}
	As well item $ (b) $, the equality could only be achieved if, $\lambda_i(\B^T\B)-k \leq 0$, for all $i=1,\ldots,m$, thus
	$$\lambda_i(\B^T\B)\leq k, \quad \Rightarrow \quad \lambda_i(\A_\lin) \leq 0.$$
	As $\mathrm{Tr}(\A_\lin) =0$, then all eigenvalues of $\A_\lin$ must be zeros, so $\A_\lin = \mathbf{0}$, i.e. $\h$ should have only isolated edges. But this contradicts the fact that the number of edges is greater than the number of vertices. Therefore this equality cannot be achieved.
	
	Finally, to prove the last inequality of item $ (c) $, we observe that	
	\begin{eqnarray}
	\notag \E(\A_\lin) &=& \sum_{i=1}^m\left|\lambda_i(\B^T\B)-k\right| = \sum_{i=1}^n\left|\lambda_i(\B\B^T)-k\right|+ k(m-n)\\
	\notag &\leq& \sum_{i=1}^n\left|\lambda_i(\B\B^T)-\frac{km}{n}\right| + \sum_{i=1}^n\left|k-\frac{km}{n}\right| + k(m-n)\\
	\notag &=& \QE(\h)+2k(m-n).
	\end{eqnarray}	
	As in the second part of item $ (b) $, equality could only be achieved, if $\lambda_i(\B\B^T)-\frac{km}{n} \geq 0$, for all $i=1,\ldots,n$, thus	$$\lambda_i(\B^T\B)\geq \frac{km}{n} > k, \quad \Rightarrow \quad \lambda_i(\A_\lin) > 0.$$
	As $\mathrm{Tr}(\A_\lin) =0$, this matrix cannot have all eigenvalues positive, so the equality cannot be achieved.
\end{proof}

\section{Power hypergraphs}\label{sec:power}
In this section, we compute the exact value of signless Laplacian energy from certain power hypergraphs. In addition, we obtain some properties of the adjacency energy of a line multigraph from a power hypergraph. For more details about this class, see \cite{Kaue-powers}.

\begin{Def}
	Let $\h=(V,E)$ be a $k$-graph, let $s \geq 1$ and $r \geq ks$ be
	integers. We define the (generalized) \textit{power hypergraph} $\h^r_s$ as
	the $r$-graph with the following sets of vertices and edges
	$$V(\h^r_s)=\left( \bigcup_{v\in V} \varsigma_v\right) \cup \left(
	\bigcup_{e\in E} \varsigma_e\right)\;\; \textrm{and}\;\;
	E(\h^r_s)=\{\varsigma_e\cup \varsigma_{v_1} \cup \cdots \cup \varsigma_{v_k}
	\colon e=\{v_1,\ldots, v_k\} \in E\},$$	where $\varsigma_{v}=\{v_{1}, \ldots,
	v_{s}\}$ for each vertex $v \in V(\h)$ and $\varsigma_e=\{v^1_e,
	\ldots,v^{r-ks}_e\}$ for each edge $e \in E(\h)$.
\end{Def}

Informally, we may say that $\h^r_s$ is obtained from a \textit{base hypergraph}
$\h$, by replacing each vertex $v\in V(\h)$ by a set $\varsigma_v$ of
cardinality $s$, and by adding a set $\varsigma_e$ with $r-ks$ new vertices
to each edge $e \in E(\h)$. For simplicity, we will denote $\h^r = \h^r_1$ and $\h_s = \h^{ks}_s$, so $\h^r_s = (\h_s)^r$.

\begin{Exe}
	The power hypergraph $(P_4)^5_2$ of the path $P_4$ is illustrated in
	Figure \ref{fig:ex1}.
	\begin{figure}[h]
		\centering
		\begin{tikzpicture}
		[scale=1,auto=left,every node/.style={circle,scale=0.9}]
		\node[draw,circle,fill=black,label=below:,label=above:\(v_1\)] (v1) at (0,0) {};
		\node[draw,circle,fill=black,label=below:,label=above:\(v_2\)] (v2) at (4,0) {};
		\node[draw,circle,fill=black,label=below:,label=above:\(v_3\)] (v3) at (8,0) {};
		\node[draw,circle,fill=black,label=below:,label=above:\(v_4\)] (v4) at (12,0) {};
		\path
		(v1) edge node[left]{} (v2)
		(v2) edge node[below]{} (v3)
		(v3) edge node[left]{} (v4);
		\end{tikzpicture}
		
		\begin{tikzpicture}
		\node[draw,circle,fill=black,label=below:,label=above:\(v_{11}\)] (v1) at (0,0) {};
		\node[draw,circle,fill=black,label=below:,label=above:\(v_{12}\)] (v11) at (1,0) {};
		\node[draw,circle,fill=black,label=below:,label=above:] (v22) at (2.25,0) {};
		
		\node[draw,circle,fill=black,label=below:,label=above:\(v_{21}\)] (v3) at (3.5,0) {};
		\node[draw,circle,fill=black,label=below:,label=above:\(v_{22}\)] (v31) at (4.5,0) {};
		\node[draw,circle,fill=black,label=below:,label=above:] (v42) at (5.75,0) {};
		
		\node[draw,circle,fill=black,label=below:,label=above:\(v_{31}\)] (v5) at (7.5,0) {};
		\node[draw,circle,fill=black,label=below:,label=above:\(v_{32}\)] (v51) at (8.5,0) {};
		\node[draw,circle,fill=black,label=below:,label=above:] (v62) at (9.75,0) {};
		
		\node[draw,circle,fill=black,label=below:,label=above:\(v_{41}\)] (v64) at (11,0) {};
		\node[draw,circle,fill=black,label=below:,label=above:\(v_{42}\)] (v7) at (12,0) {};

		\begin{pgfonlayer}{background}
		\draw[edge,color=gray] (v1) -- (v31);
		\draw[edge,color=gray,line width=25pt] (v3) -- (v51);
		\draw[edge,color=gray] (v5) -- (v7);
		\end{pgfonlayer}
		\end{tikzpicture}
		\caption{The power hypergraph $(P_4)^5_2$.}
		\label{fig:ex3}\label{fig:ex1}
	\end{figure}
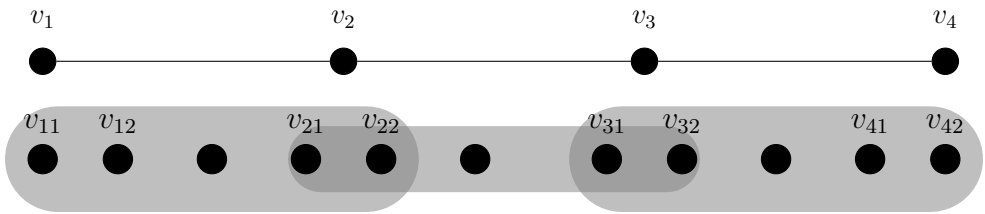
\end{Exe}

\begin{Obs}
	Let $\h$ be a $k$-graph with $n$ vertices, $m$ edges having
	signless Laplacian eigenvalues $\lambda_1\geq\lambda_2\geq\cdots\geq\lambda_t
	>\lambda_{t+1}=\cdots=\lambda_n=0$. According to \cite{Kaue-lap}, we have the following.
	
	For $r > ks$ the eigenvalues of $\Q(\h^r_s)$ are $s\lambda_1+r-ks, \ldots,
	s\lambda_t+r-ks$, and $r-ks$ with multiplicity $m-t$, and $0$ with
	multiplicity $(r-ks-1)m+sn$.
\end{Obs}

\begin{Teo}
	Let $\h$ be a $k$-graph with $n$ vertices and $m$ edges. For integers $s \geq 1$ and $r > ks$, we have that
	\begin{enumerate}
		\item[(a)] If $r-ks > d(\h^r_s)$, then $\QE(\h_s^r) = 2rm\left(1 - \frac{m}{ns+(r-ks)m} \right) > 2ksm.$
		\item[(b)] If $r-ks = d(\h^r_s)$, then $\QE(\h_s^r) =  2ksm.$
		\item[(c)] If $r-ks < d(\h^r_s)$, then $\QE(\h_s^r) < 2ksm.$
		
	\end{enumerate}
\end{Teo}
\begin{proof}
	First of all, we notice that $d(\h^r_s) = \frac{rm}{|V(\h_s^r)|}$, where $|V(\h_s^r)| = ns+(r-ks)m$. Now, let $t$ be the number of positive eigenvalues of $\Q(\h)$, thus
	
	$$\QE(\h_s^r) = \sum_{i=1}^{|V(\h_s^r)|}\left|\lambda_i(\h_s^r) - \frac{rm}{|V(\h_s^r)|} \right| =$$
	$$\sum_{i=1}^t\left|s\lambda_i(\h) + (r-ks) - \frac{rm}{|V(\h_s^r)|} \right| + \sum_{i=t+1}^m\left|(r-ks) - \frac{rm}{|V(\h_s^r)|} \right| +  \sum_{i=m+1}^{|V(\h_s^r)|}\frac{rm}{|V(\h_s^r)|}.$$
	
	For item $(a)$, let $r-ks > d(\h^r_s)$, so
	\begin{eqnarray}\notag \QE(\h_s^r) &=& \sum_{i=1}^ts\lambda_i(\h) + \sum_{i=1}^m\left((r-ks) - \frac{rm}{|V(\h_s^r)|} \right) +  \sum_{i=m+1}^{|V(\h_s^r)|}\frac{rm}{|V(\h_s^r)|}\\
	\notag &=& ksm + m(r-ks) - \frac{rm^2}{|V(\h_s^r)|} + (|V(\h_s^r)|-m)\frac{rm}{|V(\h_s^r)|}\\
	\notag &=& 2rm\left(1 - \frac{m}{|V(\h_s^r)|} \right) >  2rm\left(1 - \frac{r-ks}{r} \right) = 2ksm.
	\end{eqnarray}
	
	Now, for item $(b)$, let $r-ks = d(\h^r_s)$, so
	\begin{eqnarray}\notag \QE(\h_s^r) &=& \sum_{i=1}^ts\lambda_i(\h)  +  \sum_{i=m+1}^{|V(\h_s^r)|}\frac{rm}{|V(\h_s^r)|} = ksm + (|V(\h_s^r)|-m)\frac{rm}{|V(\h_s^r)|}\\
	\notag &=&  ksm + rm\left(1 - \frac{m}{|V(\h_s^r)|} \right) = ksm + rm\left(1 - \frac{r-ks}{r} \right) = 2ksm.
	\end{eqnarray}
	
	Finally, for item $(c)$, let $r-ks < d(\h^r_s)$, so
	\begin{eqnarray}\notag \QE(\h_s^r) &<& \sum_{i=1}^ts\lambda_i(\h) + \sum_{i=1}^m\left(\frac{rm}{|V(\h_s^r)|} -(r-ks)\right) +  \sum_{i=m+1}^{|V(\h_s^r)|}\frac{rm}{|V(\h_s^r)|}\\
	\notag &=& ksm + \frac{rm^2}{|V(\h_s^r)|} - m(r-ks) + (|V(\h_s^r)|-m)\frac{rm}{|V(\h_s^r)|} = 2ksm.
	\end{eqnarray}
	Therefore, the result follows.
\end{proof}

\begin{Lem}\label{lem:linhk}
	Let $\h$ be a $k$-graph. For integers $r \geq k$ and $s \geq 1$, we have $$\A(\lin(\h^r)) = \A(\lin(\h))\quad \textrm{and}\quad \A(\lin(\h_s)) = s\A(\lin(\h)).$$
\end{Lem}
\begin{proof}
In the first equality, we observe that the number of hyperedges, and the connections between it, do not change by adding new vertices. So, the line multigraph of $\h$ is the same as the line multigraph of $\h^r$ and consequently $\A(\lin(\h)) = \A(\lin(\h^r))$.

For the second equality, we notice that change each vertex for a set of cardinality $s$, do not change the number of hyperedges. Further, we observe that, if two hyperedges are disjoint in $ \h $, they must remain disjoint in $ \h_s $, but if two hyperedges had $t$ common vertices in the base $k$-graph, then they will have $ st $ common vertices in the power hypergraph. Therefore, if two vertices have $ t $ common edges in $ \lin(\h) $, these vertices should have $st$ common edges in $ \lin (\h_s) $. That is, $\A(\lin(\h_s)) = s\A(\lin(\h)).$
\end{proof}

\begin{Pro}\label{teo:espec-power-lin}
	Let $\h$ be a $k$-graph. For integers $s \geq 1$ and $r \geq ks$, we have
	$$P_{\lin(\h^r_s)}(\lambda) = s^mP_{\lin(\h)}(\lambda/s).$$
That is, $ \lambda $ is an eigenvalue of $\A(\lin(\h))$ if and only if $s\lambda$ is an eigenvalue of $\A(\lin(\h^r_s))$.
\end{Pro}
\begin{proof}
	We notice that
	\begin{eqnarray}
	\notag P_{\lin(\h^r_s)}(\lambda) &=& \det\left(\lambda\mathbf{I} - \A(\lin(\h^r_s)) \right) = \det\left(\lambda\mathbf{I} - \A(\lin(\h_s)) \right)\\
	\notag &=& \det\left(\lambda\mathbf{I} - s\A(\lin(\h)) \right)= s^m\det\left((\lambda/s)\mathbf{I} - \A(\lin(\h)) \right) = s^mP_{\lin(\h)}(\lambda/s).
	\end{eqnarray}
	Therefore, we conclude the result.
\end{proof}

\begin{Teo}
	Let $\h$ be a $k$-graph. If $s \geq 1$ and $r \geq ks$ are integers, then
	$$\E(\lin(\h^r_s)) = s\E(\lin(\h)).$$
\end{Teo}
\begin{proof}
	According to Proposition \ref{teo:espec-power-lin}, we have $$\E(\lin(\h^r_s)) =\sum_{i=1}^m|\lambda_i(\lin(\h^r_s))| = \sum_{i=1}^m|s\lambda_i(\lin(\h))| = s\E(\lin(\h)).$$
	Therefore, the result follows.
\end{proof}

\begin{Lem}[Lemma 4, \cite{Kaue-lap}]\label{lem:graulinha}
	Let $\h$ be a $k$-graph and $\lin(\h)$ its line graph. If $u \in
	V(\lin(\h))$ is a vertex obtained from the edge $e\in E(\h)$, then
	$d_\lin(u) = \sum_{v \in e}\left(d_\h(v)-1\right).$
\end{Lem}

\begin{Lem}\label{lem:mlin}
	Let $\h$ be a $k$-graph with $n$ vertices and $m$ edges. If $m_\lin$ is the number of edges from the line multigraph $\lin(\h)$, then $m_\lin = \frac{1}{2}\left(Z(\h)-km\right)$.
\end{Lem}
\begin{proof}
	By Lemma \ref{lem:graulinha}, we have
	\begin{eqnarray}
	\notag 2m_\lin &=& \sum_{v \in V(\lin(\h))}d_\lin(v) = \sum_{e \in E(\h)}\left(\sum_{u \in e}(d_\h(u)-1) \right) = \sum_{u \in V(\h)}d_\h(u)(d_\h(u)-1)\\
	\notag &=& \sum_{u \in V(\h)} d_\h(u)^2 - \sum_{u \in V(\h)} d_\h(u) = Z(\h)-km
	\end{eqnarray}
	Therefore, the result follows.
\end{proof}

\begin{Lem}[Theorem 5.2, \cite{energy-book}]\label{lem:bound-adj}
	If $\g$ is a graph with $m$ edges, then $2\sqrt{m} \leq \E(\g) \leq 2m$.
\end{Lem}

A hypergraph is \textit{linear} if each pair of edges has at most one common vertex.
\begin{Teo}\label{teo:indice-cota-power}
	If $\h$ is a linear $k$-graph with $n$ vertices and $m$ edges, then
	$$\sqrt{2s^2(Z(\h)-km)} \leq \E(\lin(\h^r_s)) \leq s(Z(\h)-km).$$
\end{Teo}
\begin{proof}
	First note, if $\h$ is linear, then $\lin(\h)$ is a graph, so by Lemma \ref{lem:bound-adj} we have
	$$2\sqrt{m_\lin} \leq \E(\lin(\h)) \leq 2m_\lin \quad \Rightarrow \quad 2s\sqrt{m_\lin} \leq \E(\lin(\h_s^r)) \leq 2sm_\lin.$$
	Now by Lemma \ref{lem:mlin}, we have
	$$2s\sqrt{\frac{1}{2}\left(Z(\h)-km\right)} \leq \E(\lin(\h_s^r)) \leq 2s\left( \frac{1}{2}\left(Z(\h)-km\right)\right).$$
	Thus we prove the result.
\end{proof}

\section*{Acknowledgments} This work is part of
doctoral studies of K. Cardoso under the supervision of V.~Trevisan. K.
Cardoso is grateful for the support given by Instituto Federal do Rio Grande
do Sul (IFRS), Campus Feliz. V. Trevisan acknowledges partial support of CNPq grants
409746/2016-9 and 303334/2016-9, CAPES (Proj. MATHAMSUD 18-MATH-01) and
FAPERGS (Proj.\ PqG 17/2551-0001).

\bibliographystyle{acm}
\bibliography{Bibliografia}
\end{document}